\documentclass[11pt]{amsart}
\usepackage{amssymb}
\usepackage{amscd}
\usepackage{color}
\usepackage{url}
\usepackage{hyperref}
\usepackage{xcolor}
\usepackage{graphicx}

\newenvironment{demode}
{\noindent {{\it Proof of }}}%
{\par \hfill \fbox{}}

\numberwithin{equation}{section}

\def\re{\mathop{\rm Re}\nolimits}
\def\ker{\mathop{\rm ker}\nolimits}
\def\x{\mathop{\mathbf{x}}\nolimits}

\theoremstyle{plain}
\newtheorem{theorem}{Theorem}[section]
\newtheorem{thm}[theorem]{Theorem}
\newtheorem{thm*}{Theorem}

\newtheorem{cor}[theorem]{Corollary}

\newtheorem{prop}[theorem]{Proposition}
\theoremstyle{definition}

\theoremstyle{remark}
\newtheorem{remark}[theorem]{Remark}

\newtheorem*{thRhaly}{{\textbf{\rm Theorem (Rhaly, 1989)}}}
\newtheorem*{thRhalysp}{{\textbf{\rm Theorem (Spectrum of $R_{\textbf{a}}$)}}}

\newcommand{\ol}{\overline}

\def\CC{\mathbb C}
\def\DD{\mathbb D}
\def\NN{\mathbb N}

\def\GH{\mathcal H}

\def\beginpf{\begin{proof}}
\def\endpf{\end{proof}}
\def\beq{\begin{equation}}
\def\eeq{\end{equation}}
\def\diag{\mathop{\rm diag}\nolimits}

\newcommand{\xdownarrow}[1]{%
  {\left\downarrow\vbox to #1{}\right.\kern-\nulldelimiterspace}
}

\def\ber{\begin{eqnarray*}}
\def\eer{\end{eqnarray*}}

\usepackage{xcolor}
\newcommand{\0}{{\color{lightgray}0}}

\begin{document}

\title[Rhaly operators]{Rhaly operators: more on generalized Ces\`aro operators}

\author{Eva A. Gallardo-Guti\'{e}rrez}
\address{Eva A. Gallardo-Guti\'errez \newline
Departamento de An\'alisis Matem\'atico y Matem\'atica Aplicada,\newline
Facultad de Matem\'aticas,
\newline Universidad Complutense de
Madrid, \newline
Plaza de Ciencias 3, 28040 Madrid,  Spain
\newline
and Instituto de Ciencias Matem\'aticas ICMAT,
\newline Madrid,  Spain }
\email{eva.gallardo@mat.ucm.es}

\author{Jonathan R. Partington}
\address{Jonathan R. Partington, \newline
School of Mathematics, \newline
University of Leeds, \newline
Leeds LS2 9JT, United Kingdom}
\email{J.R.Partington@leeds.ac.uk}

\thanks{Both authors are partially supported by Plan Nacional  I+D grant no. PID2022-137294NB-I00, Spain. First author is also supported by
 the Spanish Ministry of Science and Innovation, through the ``Severo Ochoa Programme for Centres of Excellence in R\&D'' (CEX2019-000904-S) and from the Spanish National Research Council, through the ``Ayuda extraordinaria a Centros de Excelencia Severo Ochoa'' (20205CEX001). }


\date{August 2024}



\begin{abstract}
Rhaly operators, as   generalizations of the Ces\`aro operator, are studied from the standpoint of view of spectral theory and invariant subspaces, extending previous results by Rhaly and Leibowitz to a framework where generalized Ces\`aro operators arise naturally.
\end{abstract}


\maketitle

\section{Introduction and preliminaries}
Let $\textbf{a}=(a_n)_{n\geq 0}$ be a sequence of complex numbers and
\begin{equation*}
R_{\textbf{a}} := \begin{bmatrix}
a_0 & \0 & \0 & \0 & \0 &  \cdots\\[3pt]
a_1 & a_1 & \0 & \0 & \0 & \cdots\\[3pt]
a_2 & a_2 & a_2 & \0 & \0 & \cdots\\[3pt]
a_3 & a_3 & a_3 & a_3 & \0 & \cdots\\[3pt]
a_4 & a_4 & a_4 & a_4 & a_4 & \cdots\\
\vdots & \vdots & \vdots & \vdots & \vdots & \ddots
\end{bmatrix}.
\end{equation*}
the associated Rhaly matrix (terraced matrix). In \cite{Rhaly1}, Rhaly showed that whenever the limit
$$
L=\lim_{n\to \infty} (n+1)|a_n|
$$
exists, it provides a test for the boundedness of $R_{\textbf{a}}$ as an operator acting on the classical sequence space $\ell^2$ (see also \cite{Rhaly2}). More precisely,
\begin{thRhaly}
\mbox{ }\newline
\vspace*{-0,3cm}
\begin{enumerate}
\item If $L< +\infty$ and $D$ denotes the diagonal operator with diagonal entries $\{(n+1)a_n :\, n =0, 1, 2, 3, \dots\}$,
then $R_{\textbf{a}}$ is a bounded operator on $\ell^2$ with
$$\|R_{\textbf{a}}\|\leq \|D\|+ \sup \{\sqrt{n(n+1)} |a_n|:\, n =0, 1, 2, 3, \dots\}.$$
Moreover,
if $L=0$, then $R_{\textbf{a}}$ is a compact operator (with the numbers $\{a_n\}_{n\geq 0}$ in the point spectra of both $R_{\textbf{a}}$ and its adjoint $R_{\textbf{a}}^*$).
\item If $L=+\infty$, then $R_{\textbf{a}}$ is not a bounded operator on $\ell^2$.
\end{enumerate}
\end{thRhaly}

The classical Ces\'aro matrix $\mathcal{C}$ appears as a particular instance of a Rhaly matrix where $\textbf{a}=(1/(n+1))_{n\geq 0}$. In particular, $\mathcal{C}$ takes a complex sequence $\textbf{x}=(x_0, x_1, x_2\dots )$ to that with $n$th entry:
$$
(\mathcal{C}\, \textbf{x})_n= \frac{1}{n+1} \sum_{k=0}^n x_k, \qquad (n\geq 0),
$$
and, identifying sequences with Taylor coefﬁcients of power series, $\mathcal{C}$ can be expressed as the integral operator acting on holomorphic functions $f(z)=\sum_{k=0}^{\infty} x_k z^k$ of the unit disc $\DD$ as
\begin{equation}\label{integral definition}
\mathcal{C}(f)(z)= \left \{ \begin{array}{ll}
\displaystyle \frac{1}{z} \int_0^z \frac{f(\xi)}{1-\xi} \, d\xi, & z\in \DD\setminus\{ 0\}, \\
\noalign{\medskip}
f(0)&  z=0;
\end{array} \right.
\end{equation}
for $z\in \DD$. There is an extensive literature on the Ces\`aro operator, and more general, on integral operators, acting on a large variety of spaces of analytic functions regarding its boundedness, compactness or spectral picture (see the recent survey by Ross \cite{ross}). In the pioneering work of \cite{BHS}, Brown, Halmos and Shields computed the spectrum of $\mathcal{C}$ in $\ell^2$ showing that it is the closed disc:
$$\sigma(\mathcal{C})=\{z\in \mathbb{C}:\; |z-1|\leq 1\},$$
while the point spectrum of $\mathcal C$ is empty and that of $\mathcal{C}^*$ the open disc
$$\sigma_p(\mathcal{C}^*)=\{z\in \mathbb{C}:\; |z-1|<1\}.$$
Likewise, they proved that $\mathcal{C}$ is hyponormal in $\ell^2$, namely, the conmutant $[\mathcal{C}^*,\mathcal{C}]$ is positive semi-definite. Later on, Kriete and Trutt \cite{KT71} established a remarkable result proving that the Ces\`aro operator is, indeed, \emph{subnormal}, namely,  $\mathcal{C}$  has a normal extension. 

In the particular instances of $\textbf{a}=(1/(n+1)^\alpha)_{n\geq 0}$ for $\alpha>1$, it turns out that all Rhaly matrices induce compact and nonhyponormal operators. In case that $\textbf{a}$ is a sequence of real numbers, Rhaly found necessary conditions for the hyponormality of $R_{\textbf{a}}$ (see \cite{Rhaly1}) but the general question of
characterizing those sequences $\textbf{a}$ inducing hyponormal operators $R_{\textbf{a}}$
or moreover, subnormal operators, remains open.

\smallskip

Regarding the spectrum of $R_{\textbf{a}}$, Rhaly \cite{Rhaly1} was able to extend Brown, Halmos and Shields' techniques to prove  the following

\begin{thRhalysp} \label{Rhaly spectrum}
Let $\textbf{a}=(a_n)_{n\geq 0}$ be a sequence of positive numbers. Assume $a_i\neq a_j$ for all $i\neq j$ and $L=\lim_{n\to \infty} (n+1)a_n$ exists and  is finite.
Then,
\begin{enumerate}
\item If $L=0$, then $\sigma_p(R_{\textbf{a}}^*)= \textbf{a}$;
\item If $L>0$, then
$$\textbf{a}\cup \{z\in \mathbb{C}:\; |z-L|< L\}\subseteq \sigma_p(R_{\textbf{a}}^*) \subseteq \textbf{a}\cup \{z\in \mathbb{C}:\; |z-L|\leq L\}\setminus \{0\} ;$$
\item $\textbf{a}\cap (2L, +\infty)\subseteq \sigma_p(R_{\textbf{a}}) \subseteq \textbf{a}\cap [2L, +\infty);$
\item $\sigma(R_{\textbf{a}})= \textbf{a}\cup \{z\in \mathbb{C}:\; |z-L|\leq L\}.$
\end{enumerate}
\end{thRhalysp}
It is worth pointing out that Leibowitz in \cite{Leibowitz} and \cite{Leibowitz2} had also studied the spectrum of particular Rhaly matrices as operators acting  on $\ell^p$-spaces, $1<p<\infty$. For the spectrum of the Ces\`aro operator in $\ell^p\ (1<p<\infty)$ see, for instance, \cite{gonzalez}.

\smallskip

As might be expected, if the sequence $\textbf{a}$  corresponds to a \emph{moment sequence} of a finite positive (or complex) Borel measure, the operator properties of $R_{\textbf{a}}$ are more approachable. Nevertheless, in this paper we shall work mostly with the assumption that $((n+1)a_n)_{n\geq 0}$ is a bounded sequence,
not necessarily tending to a limit, which is a broader scenario. Indeed, in such a case, the linear operator $D_{\textbf{a}}$
with diagonal entries $((n+1)a_n)_{n\geq 0}$ is bounded, so $R_{\textbf{a}}=D_{\textbf{a}} \mathcal{C}$ defines also a bounded operator. In general the condition that $((n+1)a_n)_{n\geq 0}$
is bounded is not necessary for the boundedness of $R_{\textbf{a}}$, as shown by the example of
Leibowitz~\cite[p.~283]{Leibowitz} with
\[
a_n=\begin{cases}
n^{-7/8} & \hbox{if $n$ is a perfect square},\\
0 & \hbox{otherwise},
\end{cases}
\]
but, it turns out to be necessary and sufficient for the boundedness of $R_{\textbf a}$ if $(a_n)$ is a moment sequence as shown by Galanopoulos, Girela and Merch\'an in \cite{GGM}.
At this regard, in both \cite{GGM, GGM2}, the authors deal with boundedness of $R_{\textbf{a}}$ as integral operator (along the lines of
\eqref{integral definition}) not only in Hardy spaces, but also in weighted Bergman spaces, BMOA or the Bloch space (for boundedness of Ces\`aro-like operators in Hardy spaces, we refer also to the previous works by Stempak \cite{Stempak} and Andersen \cite{Andersen}).

\smallskip

In this setting, the main goal of this work is twofold. On one hand, it is taking further the study of the spectrum of Rhaly operators in $\ell^2$ providing concrete spectral picture when $\textbf{a}$ is a moment sequence of a finite positive Borel measure. In particular, the results in Section \ref{Section 2} extend Rhaly's theorem on the spectrum of $R_{\textbf{a}}$ previously stated and exhibit examples of finite positive Borel measures $\mu$ such that the corresponding moment sequence $(\mu_n)_{n\geq 0}$ induces Rhaly operators $R_{\mathbf{\mu}}$ in $\ell^2$ with different spectral configuration. Likewise, we shall give a conceptually simpler proof that the \emph{numerical range} of $R_{\mathbf{\mu}}$ is always contained in the closed right half plane $\overline{\CC_+}=\{z\in \CC:\; \re z\geq 0\}$, which implies, in particular, the same containment for the spectrum $\sigma(R_{\mathbf{\mu}})$. Consequently, it follows that $R_{\mathbf{\mu}}$ is related to the infinitesimal generator of a contraction semigroup (see, for example, \cite[Cor. II.3.17, Prop. II.3.23]{EN} and \cite[Thm. IV.4.1]{SFBK}).

\smallskip

This latter fact is closely related to our second concern regarding Rhaly operators, namely, the study of their invariant subspaces along the lines of the authors' works \cite{GP24} and \cite{GPR24}. Relating the Ces\`aro operator to some semigroups, either as an infinitesimal generator or as the resolvent operator at 0,  has turned out to be  fruitful  in order to compute norms and spectra (see Siskakis' work \cite{siskakis}), determine subnormality (see Cowen's paper \cite{cowen}) or even to study local spectral properties (see the recent work \cite{GG}). In \cite{GP24} the authors linked the invariant subspaces of $\mathcal{C}$ to those of the right-shift semigroup $\{S_{\tau}\}_{\tau\geq 0}$ acting on a particular weighted $L^2(\mathbb{R}, w(y)dy)$ as an approach towards describing completely the lattice of the invariant subspaces of $\mathcal{C}$. In Section \ref{Section 3}, we undertake this study and show a connection with the invariant subspaces of a family (not necessary a semigroup) of weighted composition operators $\{W_t\}_{t\geq 0}$ acting on the classical Hardy space $H^2$.

\smallskip

A related family of weighted composition operators arise also naturally linked to the Hilbert matrix $\mathcal{H}$, as it was shown by  Diamantopoulos and Siskakis in \cite{DS00}. Finally, in Section \ref{Section 4}, we discuss a similar approach to the invariant subspaces as well as the fact that $\mathcal{H}$ and its generalizations are also related to the infinitesimal generator of a contraction semigroup.

\section{Spectrum of Rhaly operators: a step further}\label{Section 2}

In this section we study the spectrum of Rhaly operators when the induced sequence is a \emph{moment sequence}, which extends Rhaly's theorem recalled in the Introduction.
We start by determining the point spectrum, that is, the set consisting of eigenvalues.

\begin{theorem}\label{point spectra}
Let $\mu$ be a positive finite Borel measure in $[0,1)$ and $R_{\mathbf{\mu}}$ the Rhaly operator associated to the moment $(\mu_n)_{n\geq 0}$. Assume $R_{\mathbf{\mu}}$ is a bounded operator in $\ell^2$ and let
\begin{equation}\label{moment sum}
s_n=\sum_{j=0}^n \mu_j = \int_0^1 \frac{1-t^{n+1}}{1-t} \, d\mu(t).
\end{equation}
Then, for a given index $k$, $\mu_k \in \sigma_p(R_{\mathbf{\mu}})$ if and only if the sequence $(\mu_n \exp(s_n/\mu_k))_{n\geq 0}$ is in $\ell^2$. Moreover,
$$
\sigma_p(R_{\mathbf{\mu}})=\{\mu_k:\, (\mu_n \exp(s_n/\mu_k))_{n\geq 0} \in \ell^2\}.
$$
\end{theorem}

Before proceeding with the proof, note that in case $\mu$ is the Lebesgue measure $m$ in $[0,1)$, $R_{\mathbf{m}}$ is the classical Ces\`aro operator $\mathcal{C}$ and the sequence $(s_n)_{n\geq 1}$ grows as $\log n$. Thus $\mu_n \exp(s_n/\mu_k)$ behaves as $n^{-1}n^{k+1}$, and is not
an $\ell^2$ sequence for any $k$. That is, we recover the well-known
result that the point spectrum of $\mathcal{C}$ is empty, proved in \cite{BHS}.

\smallskip

Likewise, if $\mu$ is a positive finite Borel measure in $[0,1)$  such that $(\mu_n) \in \ell^1$, then $(s_n)_{n\geq 1}$ is bounded and, in such a case, every $\mu_k$ is an eigenvalue of $R_{\mathbf{\mu}}$. Clearly, taking $\mu= \delta_t$, a point mass (Dirac measure) at $t$ with $0<t<1$,  we have the moment sequence $\mu_n=t^n$, $n =0,1,2,\ldots$,
and $s_n=\frac{1-t^{n+1}}{1-t}$. Accordingly, the sequence $(\mu_n \exp(s_n/\mu_k))_n$ is $(t^n \exp \frac{1-t^{n+1}}{1-t}/t^k)$,
which is clearly in $\ell^2$, so every $\mu_k$ is an eigenvalue.

\smallskip

Besides, if we take $\mu=\delta_0 + c m$ for some $c>0$, an easy computation yields that $\mu_0=1+c$ and $\mu_n= c/(n+1)$  for $n>0$. So $s_n$ grows as $c\log n$, and for
$k=0$ the sequence $(\mu_n \exp(s_n/\mu_0))_n$ behaves as $n^{-1}n^{c/(1+c)}$, which is in $\ell^2$ for $0< c<1$.

However, for $k>0$, $(\mu_n \exp(s_n/\mu_k))_n$ behaves as $n^{-1}n^{k+1}$, which does not belong to $\ell^2$. Consequently, if  $0<c<1$ and $\mu=\delta_0 + c m$, the point spectrum of  $R_{\mathbf{\mu}}$ is simply $\{\mu_0\}=\{1+c\}$.

\smallskip

In \cite{Leibowitz}, Leibowitz  considers Ces\`aro-like matrices with
$\mu_n= (n+1)^{-s}$ for $n =0,1,2,\ldots$, where $s \ge 1$. The
case $s=1$ gives the familiar Ces\`aro operator, and for $s>1$ the corresponding  operator
is compact.

We remark that for $k \ge 0$, $n \ge 0$
\begin{eqnarray*}
\int_0^1 t^n (-\log t)^k dt &=&
\int_0^\infty e^{-nx} x^k   e^{-x} \, dx \\
&=& \int_0^\infty e^{-y} y^{k} (n+1)^{-k} \, dy/(n+1)\\
&=& \Gamma(k+1)/(n+1)^{k+1}
\end{eqnarray*}
(using the substitutions $t=e^{-x}$ and $x=y/(n+1)$),
and so for $s>1$
this corresponds to the measure $\mu$ on $(0,1)$
given by $d\mu(t)=\dfrac{1}{\Gamma(s)}(-\log t)^{s-1} \, dt$.

\smallskip

Clearly, it is possible to provide further examples of finite positive Borel measures in $[0,1)$ such that $\sigma_p(R_{\mathbf{\mu}})$ is either empty, finite or an infinite set.

\medskip

\begin{demode}{Theorem \ref{point spectra}}
We start by noting that the matrix representation of the adjoint $R_{\mathbf{\mu}}^*$ is upper triangular, and  Apostol's triangular decomposition of a bounded operator $T$ states, in particular, that if  $\ker(\lambda -T)^*\neq \{ 0\}$  then $\lambda$ belongs to the diagonal of $T$ (see \cite[Corollary 3.40, (iv)]{Herrero}, for instance). In other words, if $\ker(\overline{\lambda}-R_\mu)\neq \{ 0\}$, then $\overline{\lambda} \in (\mu_n)_{n\geq 0}$. Consequently,
$$\sigma_p(R_{\mathbf{\mu}})\subseteq (\mu_n)_{n\geq 0}.$$

\smallskip

Now, let $k\geq 0$ be fixed and $\mu_k$ the corresponding $k$-th moment of $\mu$. Assume that $\x=(x_n)_{n\geq 0}$ is an eigenvector
corresponding to the eigenvalue $\mu_k$, namely,
$$\begin{bmatrix}
\mu_0 & \0 & \0 & \0 & \0 &  \cdots\\[3pt]
\mu_1 & \mu_1 & \0 & \0 & \0 & \cdots\\[3pt]
\mu_2 & \mu_2 & \mu_2 & \0 & \0 & \cdots\\[3pt]
\mu_3 & \mu_3 & \mu_3 & \mu_3 & \0 & \cdots\\[3pt]
\mu_4 & \mu_4 & \mu_4 & \mu_4 & \mu_4 & \cdots\\
\vdots & \vdots & \vdots & \vdots & \vdots & \ddots
\end{bmatrix}
\begin{bmatrix}
x_0 \\[3pt]
x_1 \\[3pt]
x_2 \\[3pt]
x_3 \\[3pt]
x_4 \\
\vdots
\end{bmatrix}
= \mu_k \begin{bmatrix}
x_0 \\[3pt]
x_1 \\[3pt]
x_2 \\[3pt]
x_3 \\[3pt]
x_4 \\
\vdots
\end{bmatrix}.
$$
The goal is to determine whether $\x \in \ell^2$.
We have
\begin{eqnarray*}
\mu_0 x_0 &=& \mu_k x_0,\\
\mu_1 x_0+\mu_1 x_1 &=& \mu_k x_1, \\
\cdots && \cdots\\
\mu_k x_0+\mu_k x_1 + \ldots \mu_k x_k &=& \mu_k x_k,\\
\mu_{k+1}x_0+\mu_{k+1}x_1 + \ldots \mu_{k+1} x_{k+1} &=& \mu_k x_{k+1}\\
\cdots && \cdots
\end{eqnarray*}
whence $x_j=0$ for $j<k$ since the $\mu_j$ are distinct (except in the
trivial case when $\mu$ is concentrated at $0$ and $R_{\mathbf{\mu}} f= \lambda f(0)$
for some $\lambda>0$). If $x_k=0$ then $\x=0$, so let us suppose without loss of generality that
$x_k=1$.

Now for $n \ge k$ we have
\begin{eqnarray}
\mu_n(x_0+ \ldots + x_n) &=& \mu_k x_n, \label{eq:mun}\\
\mu_{n+1}(x_0+\ldots+ x_{n+1}) &=&  \mu_k x_{n+1} \label{eq:mun1}
\end{eqnarray}
and taking $\mu_n \times$ \eqref{eq:mun1} $- \mu_{n+1} \times$ \eqref{eq:mun}
we have
\[
\mu_n \mu_{k} x_{n+1} -  \mu_{n+1} \mu_k x_n = \mu_n \mu_{n+1} x_{n+1},
\]
that is,
\[
x_{n+1}= \frac{\mu_{n+1}\mu_k}{\mu_n(\mu_k-\mu_{n+1})} x_n.
\]
For easier reading we write $\alpha=\mu_k$, and then for $m \ge k$ this gives (with $x_k=1$)
\begin{equation}\label{eq:xm1}
x_{m+1}= \frac{\mu_{m+1}}{\alpha} \left( \prod_{n=k}^m \frac{\alpha}{\alpha-\mu_{n+1}} \right)=\frac{\mu_{m+1}}{\alpha} \left( \prod_{n=k}^m \frac{1}{1-\mu_{n+1}/\alpha} \right).
\end{equation}
To see whether $\x=(x_n)_{n\geq 0}$ is an $\ell^2$ sequence we note that
the $\mu_n$ tend to $0$ as $n\to \infty$ since $R_{\mathbf{\mu}}$ is a bounded operator in $\ell^2$ and $\mu_n=\textrm{O}(1/(n+1))$ (see \cite[Thm.~1, Lem.~2]{GGM}, for instance). Then
\begin{equation}\label{eq1}
(1-\mu_{n+1}/\alpha)^{-1} = \exp(\mu_{n+1}/\alpha + \textrm{O}(\mu_{n+1}^2)).
\end{equation}
From \eqref{eq:xm1} and \eqref{eq1}, noting that $(\mu_n)_{n\geq 0}$ is an $\ell^2$ sequence, it follows that if $(\mu_n \exp(s_n/\alpha))_n \in \ell^2$ then  $\x=(x_n)_{n\geq 0} \in \ell^2$ also, in which case, $\mu_k \in \sigma_p(C_\mu)$. Likewise, if $(\mu_n \exp(s_n/\alpha))_n \not \in \ell^2$, then $\x=(x_n)$ is not in $\ell^2$ either, and therefore $\mu_k \not \in \sigma_p(C_\mu)$, which concludes the proof of Theorem \ref{point spectra}.
\end{demode}

\begin{remark}
In this proof, we only used the facts that $(\mu_n)_{n\geq 0}$ are positive and distinct and $\mu_n=O(1/(n+1))$. Accordingly, it holds more generally for sequences $\textbf{a}$ under this hypotheses.
\end{remark}

Our next task is studying the point spectrum of the adjoint of Rhaly operators. First, we state the following proposition, which follows as a direct application of Apostol's triangular decomposition theorem and it does not require assuming that the induced sequence is a moment one (see \cite[Corollary 3.40]{Herrero}).

\begin{prop} \label{spectrum adjoint}
Let $\textbf{a}=(a_n)_{n\geq 0}$ be a sequence of complex numbers and $R_{\textbf{a}}$ the associated Rhaly operator acting boundedly on $\ell^2$.
Then,
\begin{enumerate}
\item[(i)] $\sigma (R_{\textbf{a}}^*)=\sigma _{l}(R_{\textbf{a}}^*)=\sigma _{lre}(R_{\textbf{a}}^*)\cup \sigma _p(R_{\textbf{a}}^*)$.
\item[(ii)] Every clopen subset of $\sigma (R_{\textbf{a}}^*)$ intersects the sequence $\textbf{a}$. Moreover, every component of $\sigma (R_{\textbf{a}}^*)$
intersects the closure of $\textbf{a}$.
\item[(iii)] Every isolated point of $\sigma (R_{\textbf{a}}^*)$ belongs to $\textbf{a}$.
\end{enumerate}
\end{prop}

Here $\sigma _{l}$ and $\sigma _{lre}$ denote the left spectrum, respectively the left--right essential spectrum (also known as the Wolf spectrum). We refer to Herrero's book \cite{Herrero} for references and properties.

\smallskip

The next theorem provides information about the point spectrum of $R_\mu^*$:

\begin{theorem}\label{point spectra adjoint}
Let $\mu$ be a positive finite Borel measure in $[0,1)$ and $R_{\mathbf{\mu}}$ the Rhaly operator associated to the moment $(\mu_n)_{n\geq 0}$. Assume $R_{\mathbf{\mu}}$ is a bounded operator in $\ell^2$ and let
$$
s_n=\sum_{j=0}^n \mu_j = \int_0^1 \frac{1-t^{n+1}}{1-t} \, d\mu(t).
$$
If $\lambda \in \sigma_p(R_{\mathbf{\mu}}^*)$ then $\lambda\neq 0$. Moreover, let $\gamma>0$ such that $\exp( - (1+\epsilon)\gamma s_n)$ is an $\ell^1$ sequence for all $\epsilon>0$. Then
$$
\{z \in \CC:|z-1/\gamma|<1/\gamma\} \subseteq \sigma_p(R_{\mathbf{\mu}}^*).
$$
\end{theorem}

The proof follows the approach of \cite{BHS} to compute the point spectrum of the adjoint of the Ces\`aro operator, though it is a bit more involved.

\begin{proof}
Suppose $\x=(x_n)_{n\geq 0} \in \ell^2$ is an eigenvector of $R_{\mathbf{\mu}}^*$ associated to  $\lambda$, namely,
$R_{\mathbf{\mu}}^* \x= \lambda \x$.
We thus have
\[
\sum_{k=m}^\infty x_k \mu_k = \lambda x_m
\]
for $m=0,1,2,\ldots$. By subtracting consecutive values of this expression, there follows the recurrence relation
$$\lambda(x_m-x_{m-1})=x_m \mu_m.$$
Clearly $\lambda=0$  gives only the trivial solution $\x=0$.

For the second half of the theorem, assume $\lambda \in \{z \in \CC: |z-1/\gamma|<1/\gamma\}$ and write $\nu=1/\lambda$.
If the sequence $(x_n)_{n\geq 0}$ satisfies
$$\begin{bmatrix}
\mu_0 & \mu_1 & \mu_2 & \mu_3 & \mu_4 &  \cdots\\[3pt]
\0 & \mu_1 & \mu_2 & \mu_3 & \mu_4 & \cdots\\[3pt]
\0 & \0 & \mu_2 & \mu_3 & \mu_4 & \cdots\\[3pt]
\0 & \0 & \0 & \mu_3 & \mu_4 & \cdots\\[3pt]
\0 & \0 & \0 & \0    & \mu_4 & \cdots\\
\vdots & \vdots & \vdots & \vdots & \vdots & \ddots
\end{bmatrix}
\begin{bmatrix}
x_0 \\[3pt]
x_1 \\[3pt]
x_2 \\[3pt]
x_3 \\[3pt]
x_4 \\
\vdots
\end{bmatrix}
= \frac{1}{\nu} \begin{bmatrix}
x_0 \\[3pt]
x_1 \\[3pt]
x_2 \\[3pt]
x_3 \\[3pt]
x_4 \\
\vdots
\end{bmatrix}.
$$
we obtain
\[
x_n=x_0 \prod_{j=1}^n (1-\mu_j \nu), \qquad n=1,2,\ldots.
\]
Since $|1/\nu-1/\gamma|<1/\gamma$, a little computation shows that $\gamma< 2 \re \nu$. So, let us write $2 \re \nu= (1+\epsilon)\gamma>0$ for some $\epsilon>0$ and argue as in \cite[pp. 130--131]{BHS}. The key inequality here is
\[
| 1- \mu_j \nu|^2
= 1 - 2 \mu_j \re\nu+ |\nu|^2\mu_j^2 \le \exp (\mu_j^2|\nu|^2 -  \mu_j(1+\epsilon)\gamma).
\]
Now, since $R_{\mathbf{\mu}}$ is bounded in $\ell^2$, the series $\sum_{n=0}^\infty \mu_n^2$ converges and the estimate
\[
|x_n|^2 \le |x_0|^2 \frac{\exp{|\nu|^2 \sum_{j=1}^n \mu_j^2 }}{\exp(s_n (1+\epsilon)\gamma)}
\]
holds. Since $\exp( - (1+\epsilon)\gamma s_n)\in \ell^1$ for all $\epsilon>0$ by hypotheses, this shows that
$(x_n)_{n\geq 0}$ is an $\ell^2$ sequence.  Accordingly, $1/nu$ belongs to the the point spectrum of $R^*_\mu$, as we wished to prove.
\end{proof}

\medskip

With Theorem \ref{point spectra adjoint} at hand, it is possible to provide examples of Rhaly operators such that the point spectra
of their adjoints differ substantially.

For instance, let $\alpha>0$ and $\mu_{\alpha}$ be the Borel measure in $[0,1)$ with density $t^{\alpha}$, that is, $d\mu_{\alpha}(t)=t^\alpha dt$. An easy computation shows that
$(\mu_\alpha)_n= 1/(n+\alpha+1)$ for $n\geq 0$. Hence, $\sum_{j=0}^n (\mu_\alpha)_j$ grows as $\log n$ and accordingly,
$$
\{z \in \CC:|z-1|<1\} \subseteq \sigma_p(R_{\mathbf{\mu_{\alpha}}}^*).
$$

\medskip

On the other hand, let us consider $\mu=\delta_0 + c m$ for some $0<c<1$, where $m$ is the Lebesgue measure in $[0,1)$.
Here $\mu_0=1+c$ and $\mu_n=c/(n+1)$ for $n>0$, and as we discussed previously, $\sigma_p(R_\mu)=\{1+c\}$.
Observe that in this case $s_n$ grows as $c \log n$, so taking $\gamma=1/c$, it follows that
$$
\{z \in \CC:|z-c|<c\} \subseteq \sigma_p(R_{\mathbf{\mu}}^*).
$$
Note that $1+c$ is not in this disc.

\begin{remark}
If $R_\mu$ is compact, clearly the spectrum does not contain a disc and it is possible to provide examples of measures such that the hypotheses of Theorem \ref{point spectra adjoint} do not hold. For example, let
$\mu$ be the Lebesgue measure on $[0,r]$ for some $0<r<1$ and zero density in  $[r,1)$. In this case,  $\mu_n=r^{n+1}/(n+1)$ and $x_n$   fail to tend to $0$ unless $\lambda=\mu_k$ for some $k$. Note that no
positive $\gamma$ satisfies the hypotheses of Theorem \ref{point spectra adjoint}.
\end{remark}

\begin{remark}
We point out that Yildirim and coauthors \cite{Yildirim, Yildirim2} have studied some spectral properties of particular Rhaly operators acting on $\ell^p$ and $c_0$. Their results in the Hilbert space setting $\ell^2$ regarding the essential spectrum follow as particular instances of Theorems \ref{point spectra} and \ref{point spectra adjoint} and Proposition \ref{spectrum adjoint} along with \cite[Thm. 4.3.18]{davies}. This latter theorem states that if $\sigma_e(T)$ denotes the essential spectrum of a bounded operator on a Banach space, and $U$ be the unbounded component of $\CC \setminus \sigma_e(T)$, then $(zI -T)$ is a Fredholm operator of zero index for all $z \in U$ and $\sigma(T) \cap U$
consists of a finite or countable set of isolated eigenvalues with finite multiplicity.
\end{remark}

We close this section by
introducing a family of weighted composition operators $(W_t)_{0 \le t < 1}$,
which we use in two ways. Here, we obtain a simple proof that the numerical range
of $R_{\mathbf{\mu}}$
(and hence also the spectrum)
 is contained in closed right half-plane $\overline{\CC_+}$.
 In the next section we shall obtain information on its invariant subspaces
 by considering the $W_t$.

\subsection{Numerical range of $R_{\mathbf{\mu}}$}
For a Hilbert space operator $T$ on $H$ the numerical
range $W(T)$ is the image of the unit sphere of $H$ under the quadratic form $x\to \langle Tx, x \rangle$ associated with the operator, namely:
\[
W(T) = \{ \langle Tx, x \rangle: x \in H, \|x\|=1 \}.
\]
It is a classical fact that  $W(T)$ lies in the closed disc of radius $\|T\|$ centered at the origin and contains all the eigenvalues of $T$.
Moreover, the spectrum of an operator lies always in the closure of its numerical range and, though the numerical range is not invariant under similarities, it is invariant under unitary similarity. But, probably, the most remarkable result about the numerical range is the Toeplitz--Hausdorff Theorem, which asserts that the numerical range is always a convex set. We refer to Bonsall and Duncan monograph \cite{BoDu} for more on the subject.

A more general form of the following result is given in \cite[Thm.~1.2]{Rhaly1}. We provide
an alternative and possibly simpler proof, introducing a decomposition method that will later be applied to Hankel operators.

\begin{theorem}\label{numerical range}
Let $\mu$ be a positive finite Borel measure in $[0,1)$ and $R_{\mathbf{\mu}}$ the Rhaly operator associated to the moment sequence $(\mu_n)_{n\geq 0}$. Assume $R_{\mathbf{\mu}}$ is a bounded operator in $\ell^2$. Then $W(R_{\mathbf{\mu}}) \subset \ol{\CC_+}$ and, consequently, $\sigma(R_{\mathbf{\mu}})\subset \ol{\CC_+}$.
\end{theorem}

In order to prove Theorem \ref{numerical range}, we recall that the classical Hardy space $H^2$ consists of holomorphic functions $f(z)=\sum_{k=0}^{\infty} x_k z^k$ in the unit disc $\DD$  with Taylor coefficients $(x_k)_{k\geq 0} \in \ell^2$. An easy computation yields that $R_{\mathbf{\mu}}$ acts on $H^2$ as the integral operator
\begin{equation}\label{integral R}
R_{\mathbf{\mu}} f(z) = \int_0^1 \frac{f(tz)}{1-tz} \, d\mu(t),
\end{equation}
and it is bounded if and only if $\mu_n=\textrm{O}(1/(n+1))$ (see \cite{GGM}, for instance).

\smallskip

Clearly, \eqref{integral R} can be expressed as
\begin{equation}\label{integral weighted}
R_{\mathbf{\mu}} f(z) = \int_0^1  W_t f(z) \, d\mu(t),
\end{equation}
where $W_t$ (for $0 \le t < 1$) is the weighted composition operator
\[
W_t f(z)=\frac{f(tz)}{1-tz}.
\]
It is not difficult to prove that for each $0 \le t < 1$,  $W_t$ is a bounded operator in $H^2$.
Moreover, for $0 < t < 1$, $W_t$ is a Rhaly operator associated to $\delta_t$, a point mass at $t$, and since $\mu_n=t^n=o(1/n)$, it is a compact   operator (see \cite{GGM}, for instance). In such a case, the  spectrum is just the closure of its set of eigenvalues; that is,
\[
\sigma(W_t)=\sigma(W_t^*)= \{0\} \cup \{t^n: n \in \NN_0\}.
\]
and
\[
\sigma_p(W_t)=\sigma_p(W_t^*)=\{t^n: n \in \NN_0\}.
\]

\smallskip

We proceed now with the proof of Theorem \ref{numerical range}.

\smallskip

\smallskip

\begin{demode}{Theorem \ref{numerical range}}
Without loss of generality, we may argue in $H^2$ and consider the expression \eqref{integral weighted} for $R_{\mathbf{\mu}}$.
Hence,  an argument involving Fubini's Theorem yields
\[
\langle R_{\mathbf{\mu}} f,f \rangle = \int_0^1 \langle W_t f,f \rangle \, d\mu(t).
\]
Accordingly, it is sufficient to show that the numerical range of each weighted composition operator $W_t$ lies in $\ol{\CC_+}$.

\smallskip

Fix $t \in [0,1)$. It is sufficient to check that
$\re \langle W_t f,f \rangle \ge 0$ for all $f$ of the form
$f(z)=(1-z)g(z)$ with $g \in H^2$; these form a  dense set since $1-z$ is an outer function (see \cite{Du} or \cite{Ru}, for instance, for the \emph{inner-outer factorization} of Hardy functions).
So
\[
 \langle W_t f,f \rangle = \langle g(tz), (1-z)g(z) \rangle.
 \]

Let us write $T$ for the self-adjoint (diagonal) operator $Tg(z)=g(t^{1/2}z)$ and $S$ for the
shift operator on $H^2$; note that $TS=t^{1/2}ST$ (each sends $z^n$ to $t^{(n+1)/2}z^{n+1}$).
Then
\ber
\re \langle W_t f,f \rangle &=&   \re\langle T^2g,  g -Sg \rangle \\ &=& \langle Tg,Tg \rangle-\re \langle Tg, t^{1/2}STg \rangle\\
& \ge & \|Tg\|^2-t^{1/2}\|Tg\|^2
 \ge  0,
 \eer
which shows $W(W_t) \subset \ol{\CC_+}$ for every $t \in [0,1)$, as we wished. From here, the statement of the theorem follows.
\end{demode}

As a consequence, it follows from \cite[Cor. II.3.17, Prop. II.3.23]{EN} and \cite[Thm. IV.4.1]{SFBK}), the following surprising fact, which does not seem to have been previously observed, regarding Rhaly operators induced by positive finite Borel measures in $[0,1)$:

\begin{cor}\label{infinitesimal generator}
If $\mu$ is a positive finite Borel measure in $[0,1)$ and $R_{\mathbf{\mu}}$ is the Rhaly operator associated to the moment sequence $(\mu_n)_{n\geq 0}$ acting boundedly in $\ell^2$, then the operator $-R_{\mathbf{\mu}}$ is the infinitesimal generator of a contraction semigroup.
\end{cor}

We observe that it is possible to compute the expression of the  contraction semigroup at least \emph{formally}. Nevertheless, we will make use of \eqref{integral weighted} in the next section to study the invariant subspaces of Rhaly operators.

\section{Invariant subspaces of Rhaly operators}\label{Section 3}

The main aim of this section is studying the invariant subspaces of Rhaly operators. Our starting point is the following result proved in \cite{GP24} for the classical Ces\`aro operator.

\begin{thm}[\cite{GP24}]\label{thm semigroup}
Let $\{\varphi_t\}_{t\ge 0}$ be the holomorphic self-maps of $\DD$ given by
\begin{equation}\label{semigroup C}
\varphi_t(z)= e^{-t}z + 1 - e^{-t}, \qquad (z\in \DD).
\end{equation}
A closed subspace $M$ in $H^2$ is invariant under the Ces\`aro operator if and only if its orthogonal complement $M^{\perp}$ is invariant under the semigroup of composition operators $\{C_{\varphi_t} \}_{t\geq 0}$.
\end{thm}

The one-parameter family $\{\varphi_t\}_{t\ge 0}$  is, indeed, a \textit{holomorphic flow} (or \textit{holomorphic semiflow} by some authors), namely, a continuous
family that has a semigroup property with respect to composition (see the monograph \cite{BCD} for a detailed account of holomorphic flows).

Along these lines, the following result holds for any Rhaly operator when  considered as acting on $H^2$:

\begin{theorem} \label{Rhaly invariant}
Let $\textbf{a}=(a_n)_{n\geq 0}$ be a complex sequence
such that the sequence $((n+1)a_n)_{n \ge 0}$ is bounded, and $R_{\textbf{a}}$ the associated Rhaly operator acting boundedly on the Hardy space $H^2$.
Let $D_{\textbf{a}}$ be  the   operator with diagonal matrix $\diag((n+1)\,a_n)_{n \ge 0}$
with respect to the standard orthonormal basis in $H^2$. Then every closed subspace $M$ invariant under $C_{\varphi_t}D_{\overline{\textbf{a}}}$ for all $t\geq 0$ is invariant under $R_{\textbf{a}}^*$.
\end{theorem}

\begin{proof}
Note that the adjoint of the Ces\`aro operator acting on the orthonormal basis $\{z^n\}_{n\geq 0}$ of $H^2$
\[
\mathcal{C}^*: z^n \to \dfrac{1}{n+1} \dfrac{1-z^{n+1}}{1-z}
\]
and the adjoint of $R_{\textbf{a}}$ acting also on $\{z^n\}_{n\geq 0}$
\[
R_{\textbf{a}}^*:  z^n \to \overline{a_n} \, \dfrac{1-z^{n+1}}{1-z},
\]
are related by means of the diagonal operator $D_{\textbf{a}}$ as follows:
$$\mathcal{C}^* D_{\textbf{a}}^*= \mathcal{C}^* D_{\overline{\textbf{a}}}= R_{\textbf{a}}^*,$$
or, equivalently,
$$D_{\textbf{a}} \mathcal{C}=R_{\textbf{a}}.$$

Now, we make use of the fact that $\mathcal{C}^*$ can be expressed in terms of the $C_0$-semigroup $\{C_{\varphi_t}\}_{t\geq 0}$ (see \cite{GP24}):
$$
\mathcal{C}^* f(z)=\int_0^\infty e^{-t} C_{\varphi_t} f(z) \, dt, \qquad (f\in H^2), 
$$
to express
$$R_{\textbf{a}}^* f(z)= \int_0^\infty e^{-t} C_{\varphi_t} D_{\overline{\textbf{a}}} f(z) \, dt, \qquad (f\in H^2). $$
From here it follows that the common invariant subspaces of $C_{\varphi_t} D_{\overline{\textbf{a}}}$ are invariant
subspaces of $R_{\textbf{a}}^*$.
\end{proof}

\smallskip

When the Rhaly operator is induced by  a positive finite Borel measure $\mu$ in $[0,1)$, equation \ref{integral R} yields that $R_{\mathbf{\mu}}$ acts on $H^2$ as
$$
R_{\mathbf{\mu}} f(z) = \int_0^1 \frac{f(tz)}{1-tz} \, d\mu(t),  \qquad (f\in H^2),
$$
or in terms of the family of weighted composition operators $\{W_t\}_{0\le t<1}$, where $W_t f(z)=\frac{f(tz)}{1-tz}$:
\[
R_{\mathbf{\mu}} f(z) = \int_0^1  W_t f(z) \, d\mu(t), \qquad (f\in H^2).
\]
Upon changing variables $t=e^{-x}$, we have
that is,
\[
R_{\mathbf{\mu}} f(z) = \int_0^\infty  \widetilde W_x f(z) \, d\nu(x),
\]
where now
\begin{equation}\label{weighted 2}
\widetilde W_x f (z)= \frac{f(e^{-x}z)}{1-e^{-x}z} \qquad (0 < x \le \infty),
\end{equation}
and $d\nu(x)=d\mu(e^{-x})$.

Though the family $\{\widetilde W_x\}_{x>0}$ does not form a semigroup, we may still conclude the following:

\begin{prop} \label{invariant common}
Let $\mu$ be a positive finite Borel measure $\mu$ in $[0,1)$ and $R_{\mathbf{\mu}}$ the associated Rhaly operator acting boundedly on the Hardy space $H^2$.
Let $\{\widetilde W_x\}_{x>0}$ be the family of weighted composition operators in $H^2$ defined by \eqref{weighted 2}. If $M$ is a closed subspace invariant under
every $\widetilde W_x$, then $M$ is invariant under $R_{\mathbf{\mu}}$.
\end{prop}

\subsection{Common invariant subspaces for $W_t$}
Motivated by Proposition \ref{invariant common}, we characterize the common invariant subspaces for the family $\{W_t\}_{0\le t<1}$.

\begin{theorem}\label{invariant W_t}
The only nonzero closed subspaces invariant under every $W_t$, $0\le t<1$, in the Hardy space $H^2$ are $z^k H^2$ for some $k \ge 0$.
\end{theorem}

\begin{proof}
First, we claim that for each $f \in H^2$ and $t \in (0,1)$ the function $f(tz)$ is in
the norm closed span of $W_x W_{t/x} f$ for $t<x<1$.
Note that,
\[
W_x W_{t/x} f(z) = W_x f (tz/x)/(1-tz/x) = \frac{f(tz)}{(1-tz)(1-xz)}.
\]
Now the closed span of the functions $1/((1-tz)(1-xz))$ for $t < x<1$  is the whole space $H^2$,
since for $g \in H^2$, if $g$ is orthogonal to these functions, then
\begin{eqnarray*}
\left\langle g, \frac{1}{(1-tz) (1-xz) }\right\rangle &=& \frac{1}{x-t}
\left(\langle g,-t/(1-tz)\rangle + \langle g,x/(1-xz) \rangle\right)\\
&=& \frac{1}{x-t} (-t g(t)+x g(x))=0,
\end{eqnarray*}
which, by the isolated zeros theorem, tells us that $zg=0$ and hence $g=0$.
Thus since $f(tz)$ is in $H^\infty$, we have that $f(tz)$ is in the closed span
of all the $W_x W_{t/x} f$, as claimed.

\smallskip

Hence, it follows that any subspace invariant under all the $W_t$ is also invariant under
all the composition operators  $C_{\phi_t}$ induced by $\phi_t(z)=tz$.

\smallskip

Now each operator $ C_{\phi_t}  $ is diagonal with respect to the usual
orthonormal basis $(z^n)_{n \ge 0}$, and since it is compact its invariant subspaces are spanned
by sets of eigenfunctions (see Chapter 1 of \cite{RR}).

Thus any common invariant subspace $M$ for $f \mapsto f(xz)/(1-xz)$, $t < x < 1$
is spanned by monomials (in fact it is sufficient that $M$ is invariant under
$f \mapsto f(x_nz)/(1-x_nz)$ for a sequence $(x_n)$ decreasing to $t$).
But with $f(z)=z^k$ we have
\[
f(xz)/(1-xz)=x^kz^k(1+xz+x^2z^2+ \ldots),
\]
from which Theorem \ref{invariant W_t} follows.
\end{proof}

\begin{remark} From Theorem \ref{invariant W_t} it follows that the only subspaces invariant under
every Rhaly operator $R_{\mathbf{\mu}}$ have the same form. Taking $\mu$ to be a delta point
mass at~$t$, one has that $R_{\mathbf{\delta}_t}=W_t$. Likewise, if $e_k(z)=z^k$, note that
\[
R_{\mathbf{\mu}}(e_k)(z) = z^k\int_0^1 \frac{ t^k}{1-tz} \, d\mu(t),
\]
so these spaces are also invariant under each $R_{\mathbf{\mu}}$.
\end{remark}

\section{A final remark: the Hilbert matrix }\label{Section 4}

In \cite{DS00}, Diamantopoulos and Siskakis observe that the Hilbert
matrix
\[
H= \left( \frac{1}{i+j+1} \right)_{i,j=0,1,2,\ldots}
\]
corresponds to a an operator on $H^2$ defined by
\[
\GH(f)(z)= \int_0^1 T_t (f)(z) \, dt,
\]
where $T_t$ is the weighted composition operator defined by
\beq\label{eq:formulads}
T_t(f)(z)=\frac{1}{(t-1)z+1} f\left(\frac{t}{(t-1)z+1} \right).
\eeq
Arguing as in Section \ref{Section 3}, it is possible to generalize this, by taking $\mu$ a positive Borel measure on $[0,1)$
and $V_t$ the weighted composition operator defined by
\begin{equation}\label{def Vt}
V_t f(z)= \frac{f(t)}{1-tz}.
\end{equation}
Then we note that  $\GH_\mu$ defined by
\[
\GH_\mu f(z) = \int_0^1 V_t (f) (z) \, d\mu(t)
\]
takes the function $e_n: z \mapsto z^n$ into
\[
\GH_\mu e_n(z) = \int_0^1 \frac{t^n} {1-tz} \, d\mu(t) =  \mu_n+ \mu_{n+1}  z+   \mu_{n+2} z^2 + \ldots,
\]
where, as before
\[
\mu_n = \int_0^1 t^n \, d\mu(t),
\]
and it therefore corresponds to the Hankel matrix
\begin{equation*}
\GH_\mu= \begin{bmatrix}
\mu_0 & \mu_1 & \mu_2 & \mu_3 & \mu_4 &  \cdots\\[3pt]
\mu_1 & \mu_2 & \mu_3 & \mu_4 & \mu_5 &  \cdots\\[3pt]
\mu_2 & \mu_3 & \mu_4 & \mu_5 & \mu_6 &  \cdots\\[3pt]
\mu_3 & \mu_4 & \mu_5 & \mu_6 & \mu_7 &  \cdots\\[3pt]
\mu_4 & \mu_5 & \mu_6 & \mu_7 & \mu_8 &   \cdots\\
\vdots & \vdots & \vdots & \vdots & \vdots & \ddots
\end{bmatrix}.
\end{equation*}

The formula in \eqref{eq:formulads} was used to calculate the norm of the original Hilbert matrix. The following theorem is similar to those stated for Rhaly operators $R_{\mathbb{\mu}}$:

\begin{theorem}\label{numerical range Hilbert matrix}
Let $\mu$ be a positive finite Borel measure in $[0,1)$ and $\GH_\mu$ the associated Hankel operator. Suppose that $\GH_\mu$ is a bounded operator in $\ell^2$. Then $W(\GH_\mu) \subset \ol{\CC_+}$. Consequently, $\sigma(\GH_\mu)\subset \ol{\CC_+}$ and the operator $-\GH_\mu$ is the infinitesimal generator of a contraction semigroup.
\end{theorem}

\begin{proof}
Reasoning as in Theorem \ref{numerical range}, it is sufficient to show that the numerical range of each rank-1 operator $V_t$, with $\mu_n=t^n$ for each $n$,  lies in $\ol{\CC_+}$.
Arguing with $\GH_\mu$ acting on the Hardy space $H^2$, we observe that for every Hardy function $f$:
\[
\langle V_t f,f \rangle = \left\langle \frac{f(t)}{1-tz},f(z) \right\rangle =
f(t)\ol{\left\langle f(z),\frac{1}{1-tz}\right  \rangle}= |f(t)|^2 \ge 0,
\]
using the properties of reproducing kernels in $H^2$.
\end{proof}

Finally, regarding common invariant subspaces of the rank-1 Hankel operators $\{V_t\}_{0\le t<1}$,  we note that

\begin{theorem}\label{invariant V_t}
Let $\mu$ be a positive finite Borel measure in $[0,1)$ and $\GH_\mu$ the associated Hankel operator. Assume $\GH_\mu$ is a bounded operator in $\ell^2$.
Let $V_t$ be the weighted composition operator in $H^2$ given by \eqref{def Vt}. Then, the only closed subspaces invariant under every $V_t$, $0\le t<1$, in the Hardy space $H^2$ are the trivial ones.
\end{theorem}

\begin{proof}
Let $M$ be a common invariant subspace for all the $V_t$, $0\le t<1$. Then if there is a function
$f \in M$ with $f(t) \ne 0$, we must have $1/(1-tz) \in M$.

Now for $f \ne 0$ in $M$, we cannot have $f(t)=0$ for more than
a countable set $S \subset [0,1)$.
Then $M$ contains the closed span of $1/(1-tz)$ for $t \in [0,1) \setminus S$, which is already
the whole of $H^2$, since   $g \in H^2$ is orthogonal to $1/(1-tz)$
if and only if $g(t)=0$.
\end{proof}

\end{document}